\documentclass[11pt]{amsart}

\usepackage{amsfonts, amstext, amsmath, amsthm, amscd, amssymb}

\usepackage{graphicx, color}

\usepackage{microtype}

\usepackage[hidelinks,pagebackref,pdftex]{hyperref}

\newtheorem{theorem}{Theorem}[section]
\newtheorem{proposition}[theorem]{Proposition}
\newtheorem{lemma}[theorem]{Lemma}
\newtheorem{corollary}[theorem]{Corollary}

\newtheorem*{namedtheorem}{\theoremname}
\newcommand{\theoremname}{testing}
\newenvironment{named}[1]{\renewcommand{\theoremname}{#1}\begin{namedtheorem}}{\end{namedtheorem}}

\theoremstyle{definition}

\newtheorem{question}[theorem]{Question}

\title[Torus Lorenz Links obtained by full twists]{Torus Lorenz Links obtained by Full Twists along Torus Links}
\author{Thiago de Paiva}
\address[]{School of Mathematics, Monash University, VIC 3800, Australia }
\email[]{thiago.depaivasouza@monash.edu}

\begin{document}

\begin{abstract}
All knots are known to be hyperbolic, satellite, or torus knots, and one important family is Lorenz links, or T-links, which arise from dynamics. However, it remains difficult to determine the geometric type of a Lorenz link from a description via dynamics or as a T-link. In this paper, we consider those T-links that are torus links. 
We show that T-links obtained by full twists along torus links can never be torus links, aside from a family of cases. This addresses a question of Birman and Kofman.
\end{abstract} 

\maketitle

\section{Introduction}

\emph{Lorenz links} are knotted closed periodic orbits in the flow of the Lorenz system, which is a system of three ordinary
differential equations discovered by meteorologist E. N. Lorenz when he was trying to find equations that model weather patterns~\cite{Lorenz}.

Birman and Kofman found that \emph{T-links} exactly coincide with Lorenz links~\cite[Theorem~1]{newtwis}.
For $2\leq r_1< \dots < r_k$, and all $s_i>0$, the T-link $T((r_1,s_1), \dots, (r_k,s_k))$ is defined to be the closure of the following braid
\[ (\sigma_1\sigma_2\dots\sigma_{r_1-1})^{s_1}(\sigma_1\sigma_2\dots\sigma_{r_2-1})^{s_2}\dots(\sigma_1\sigma_2\dots\sigma_{r_k-1})^{s_k}.\]
Here $\sigma_i$ is a standard generator of the braid group $B_{r_k}$.

Thurston showed that any non-trivial knot in the 3-sphere is either a \emph{torus knot}, a \emph{satellite knot}, or has \emph{hyperbolic complement} \cite{Thurston}. These are called the \emph{geometric types} of the knot.

Birman and Kofman asked the following question ~\cite[Question 6]{newtwis}.

\begin{question}\label{question}
When are Lorenz links torus links?
\end{question}

Recall that a \emph{torus link} is a link that can be embedded on the surface of an unknotted torus
$T$ in $S^3$. A torus link is denoted by $T(p,q)$ with $p, q$ integers. Here, $p$, $q$ denote the number of times that $T(p, q)$ wraps around the meridian, longitude, respectively, of $T$. The torus link $T(p, q)$  has $\gcd(p, q)$ link components. Thus, if $\gcd(p, q) = 1$, then $T(p, q)$ is a knot.

The purpose of this paper is to answer Question~\ref{question} for T-links obtained by full twists along $(p, q)$-torus links. 
The answer to this question was not previously considered in general for this family of T-links. However, as a consequence of Lee's work on the geometric classification of twisted torus knots, it is known when a T-knot of the form $T((r, sr), (p, q))$ is a torus knot for $p$, $q$ coprime \cite[Theorem 1.1]{LeeTorusknotsobtained}, \cite[Theorem 1.1]{Positively}. Therefore, when the T-link $T((a_1, a_1s_1), \dots, (a_n,a_ns_n), (p,q))$ is a knot we assume that $n>1$.

Our main theorem is the following.

\begin{theorem}\label{maintheorem}
Let $p,q, a_1, \dots, a_n, s_1, \dots, s_n$ be positive integers such that $1<q<p$ and $1<a_1<\dots<a_n<p$ with $a_i\neq q$.
\begin{itemize}
\item If $gcd(p, q)>1$, then the T-link $$T((a_1, s_1a_1), (a_2, s_2a_2), \dots, (a_n, s_na_n), (p, q))$$ is never a torus link;

\item Otherwise, if $q< a_n$, or $p \neq bq +1$, or $s_1 > 1$, or $a_2 \neq a_1 +1$ for $b>0$, then the T-knot $$T((a_1, s_1a_1), (a_2, s_2a_2), \dots, (a_n, s_na_n), (p, q))$$ is never a torus knot for $n>1$.
\end{itemize}
\end{theorem}

Therefore, Theorem~\ref{maintheorem} answers Question~\ref{question} for T-links obtained by full twists along $(p, q)$-torus links, aside from a family of cases.

The answer to Question~\ref{question} is also related to the geometric classification of T-links. 
The geometric classification of T-links started with the study of the geometric classification of twisted torus knots. The twisted torus knots that are T-knots have the form $T((r,sr), (p, q))$ for $p, q$ coprime. 
See \cite{dePaiva:Unexpected}, \cite{LeeTorusknotsobtained}, \cite{Positively}, \cite{hyperbolicity}, and \cite{LeeThiago} for more information about the geometric classification of twisted torus knots.
Recently Thiago de Paiva and Jessica Purcell extended these results to general T-knots obtained by full twists along $(p, q)$-torus knots \cite{twisting}, \cite{dePaivaPurcell:SatellitesLorenz}.  More precisely,
they proved that there exists $B>>0$ such that if each $s_i>B$, the T-knot $T((a_1, a_1s_1), \dots, (a_n,a_ns_n), (p,q))$ is hyperbolic if and only if there is at least one $a_i>q$ that is not a multiple of $q$ \cite{twisting}. But, they didn't make  $B$ explicit. However, Theorem~\ref{maintheorem} will help to find this $B$ as they are either hyperbolic or satellite knots if they are not torus knots.


\subsection{Acknowledgment}I thank Jessica Purcell, my supervisor, and Sangyop Lee for some helpful tips.

\section{Main section}
In this section we prove Theorem~\ref{maintheorem}. We start by proving some lemmas that will be used later to prove this theorem.

Due to \cite[Lemma 2.7]{dePaivaPurcell:SatellitesLorenz}, we assume that $a_1, \dots, a_n$ are different from $q$ for the T-link
$T((a_1, s_1a_1), \dots, (a_n, s_na_n), (p, q))$.

\begin{lemma}\label{isopoty}
Consider $q<a_n$. Then, there is an ambient isotopy  that takes the T-link $$K = T((a_1, s_1a_1), \dots, (a_n, s_na_n), (p, q))$$ to an equivalent link given by the closure of the braid
\begin{align*}
&(\sigma_{a_n-1}\dots \sigma_{a_n-q+1})^{p-a_n}(\sigma_1\dots \sigma_{a_1-1})^{s_1a_1} \dots (\sigma_1\dots \sigma_{a_{n-1}-1})^{s_{n-1}a_{n-1}} \\
&(\sigma_1\dots \sigma_{a_n-1})^{s_na_n+q},  
\end{align*}
where this isotopy happens in the complement of the braid axis of the last braid.
Furthermore, $K$ has braid index equal to $a_n$.
\end{lemma}

\begin{proof}
We start with the standard braid of $K$:
$$(\sigma_1\dots \sigma_{a_1-1})^{s_1a_1} \dots (\sigma_1\dots \sigma_{a_n-1})^{s_na_n}(\sigma_1\dots \sigma_{p-1})^{q}.$$

If we pull the $(a_n-q+1)$-st strand around the braid closure anticlockwise, we see that it first reaches the bottom of the braid then passes through the sub-braid $(\sigma_1\dots \sigma_{p-1})^{q}$ and finally ends in the $(a_n+1)$-st strand as $q<a_n$.
So the $(a_n-q+1)$-st strand is connected to the $(a_n+1)$-st strand by a strand that goes one time around the braid closure. See the red strand in Figure~\ref{T1}. As this strand is an under strand, we can push it down and shrink it to reduce one strand from the last braid, as shown in Figure~\ref{T1}. After that, this strand becomes an under strand between the $(a_n-q+1)$ and the $(a_n+1)$-st strand. Thus, we obtain the braid
$$B = (\sigma_{a_n-1}\dots \sigma_{a_n-q+1})(\sigma_1\dots \sigma_{a_1-1})^{s_1a_1} \dots (\sigma_1\dots \sigma_{a_n-1})^{s_na_n}(\sigma_1\dots \sigma_{p-2})^{q}$$
after this isopoty.
\begin{figure}\label{T1}
\includegraphics[scale=0.23]{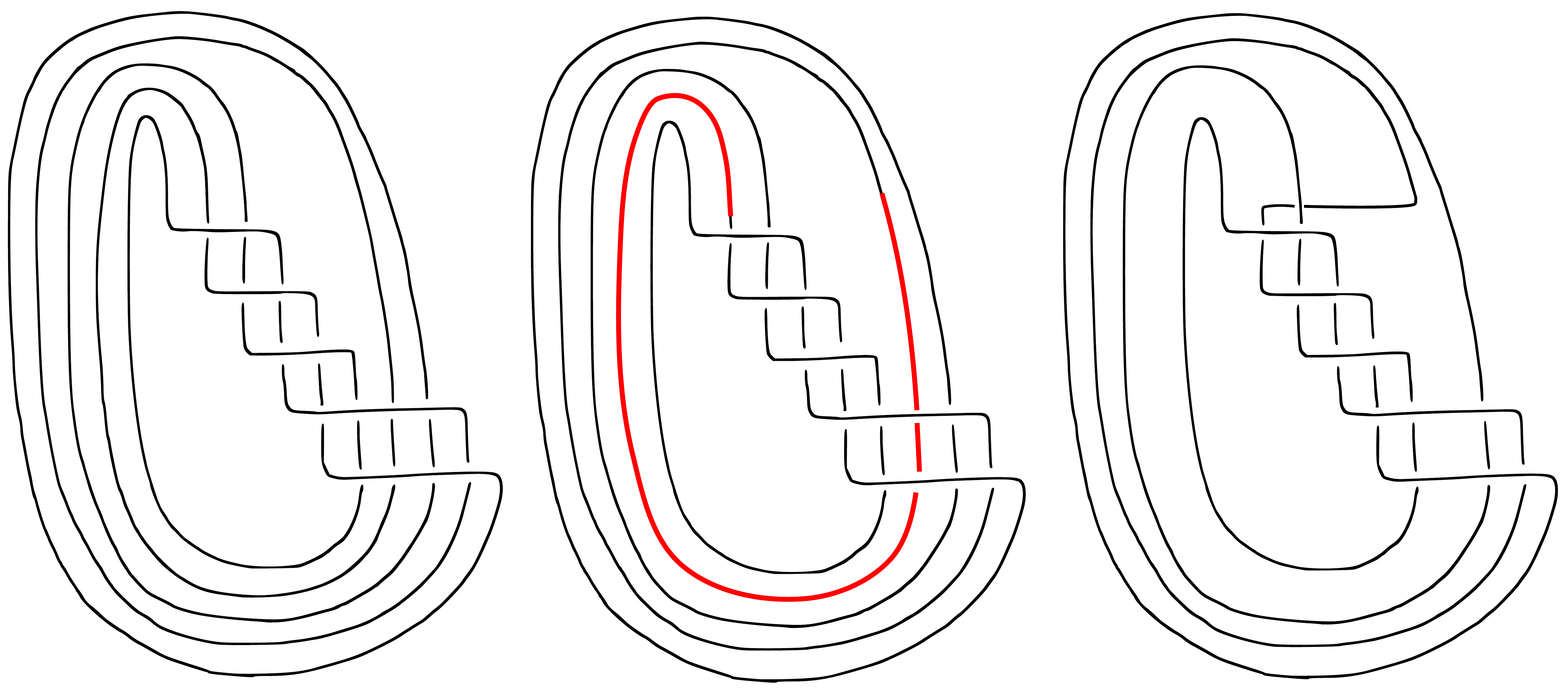} 
  \caption{The red strand in the second drawing can be replaced by a strand that goes below the third strand, as illustrated in the third drawing. This isopoty generalizes the isopoty in Figure 6 of \cite{Positively}.}
  \label{T1}
\end{figure}

As the sub-braid $(\sigma_{a_n-1}\dots \sigma_{a_n-q+1})$ of $B$ is a braid on the leftmost $a_n$ strands, the $(a_n-q+1)$-st strand is still connected to the $(a_n+1)$-st strand of $B$ by an under strand that goes one time around the braid closure.
Then we can apply the last isopoty again. After that, we obtain the braid
$$(\sigma_{a_n-1}\dots \sigma_{a_n-q+1})^2(\sigma_1\dots \sigma_{a_1-1})^{s_1a_1} \dots (\sigma_1\dots \sigma_{a_n-1})^{s_na_n}(\sigma_1\dots \sigma_{p-3})^{q}.$$
Therefore, after repeating this procedure $p-a_n-2$ more times, we finally obtain the braid
\begin{align*}
&(\sigma_{a_n-1}\dots \sigma_{a_n-q+1})^{p-a_n}(\sigma_1\dots \sigma_{a_1-1})^{s_1a_1} \dots (\sigma_1\dots \sigma_{a_{n-1}-1})^{s_{n-1}a_{n-1}} \\
&(\sigma_1\dots \sigma_{a_n-1})^{s_na_n+q}.  
\end{align*}

As $s_na_n+q> a_n$, the last braid has at least one full twist on the $a_n$ strands. So, it follows from \cite[Corollary 2.4]{Franks} that the link $K$ has braid index equal to $a_n$.
\end{proof}

\begin{lemma}\label{braidindex}
Consider $a_n < q < p$. Then, the T-link $$K = T((a_1, s_1a_1), \dots, (a_n, s_na_n), (p, q))$$ has braid index equal to $q$ and
is equivalent to the T-link $$K' = T((a_1, s_1a_1), \dots, (a_n, s_na_n), (q, p)).$$
\end{lemma}
\begin{proof}
As $a_n < q < p$, the T-link $K$ is equivalent to $K'$ \cite[Corollary 3]{newtwis}. Since $p>q$,  $K'$ has braid index equal to $q$ by  \cite[Corollary 2.4]{Franks}. 
\end{proof}

The next lemma was proved by Los \cite[Corollary 1.2]{Los}. See also \cite{de2022hyperbolic} for a quick proof sketch using the Markov theorem.
 
\begin{lemma}\label{Los}
Consider $\beta_1$, $\beta_2$ two braids which have minimal braid index and are representations of the same torus
link. Suppose that these two closed braids travel around the same braid axis $C$. Then, there is an isotopy in the complement of $C$ that takes $\beta_1$ to $\beta_2$.
\end{lemma}
 


\begin{lemma}\label{q<a_n}
Consider $q< a_n <p$. Then, the T-link $$K = T((a_1, s_1a_1), \dots, (a_n, s_na_n), (p, q))$$ is not a torus link.
\end{lemma}

\begin{proof}
We prove this lemma by induction. 

For when $n= 1$, consider that $K = T((a_1, s_1a_1), (p, q))$ is a torus link. By Lemma~\ref{isopoty}, $K$
is equivalent to the link given by the closure of the braid
$$B_1 = (\sigma_{a_1-1}\dots \sigma_{a_1-q+1})^{p-a_1}(\sigma_1\dots \sigma_{a_1-1})^{s_1a_1+q}.$$
and has braid index equal to $a_1$. So, $K$ is the $(a_1, k)$-torus link with $k>0$.

By Lemma~\ref{Los}, the link $K_2$ given by the closure of the braid
$$B_2 = B_1(\sigma_{a_1-1}^{-1}\dots \sigma_{1}^{-1})^{s_1a_1} = (\sigma_{a_1-1}\dots \sigma_{a_1-q+1})^{p-a_1}(\sigma_1\dots \sigma_{a_1-1})^{q}$$ is the $(a_1, k - s_1a_1)$-torus link.
The braid $B_2$ represents the $(p, q)$-torus link. Thus, $a_1$ is equal to $p$ or $q$, which is a contradiction.

Now we can consider that all T-links $T((a_1, s_1a_1), \dots, (a_i, s_ia_i), (p, q))$ with $q<a_i$ are not torus links for $i<n$.

Suppose that $K = T((a_1, s_1a_1), \dots, (a_n, s_na_n), (p, q))$ is a torus link. By Lemma~\ref{isopoty}, $K$ is equivalent to the link given by the closure of the braid 
\begin{align*}
&(\sigma_{a_n-1}\dots \sigma_{a_n-q+1})^{p-a_n}(\sigma_1\dots \sigma_{a_1-1})^{s_1a_1} \dots (\sigma_1\dots \sigma_{a_{n-1}-1})^{s_{n-1}a_{n-1}} \\
&(\sigma_1\dots \sigma_{a_n-1})^{s_na_n+q}  
\end{align*}
and has braid index equal to $a_n$. So, $K$ is the $(a_n, k)$-torus link with $k>0$.
By Lemma~\ref{Los}, the link $K_1$ given by the closure of the braid
\begin{align*}
&B_1 = (\sigma_{a_n-1}\dots \sigma_{a_n-q+1})^{p-a_n}(\sigma_1\dots \sigma_{a_1-1})^{s_1a_1} \dots (\sigma_1\dots \sigma_{a_{n-1}-1})^{s_{n-1}a_{n-1}} \\
&(\sigma_1\dots \sigma_{a_n-1})^{q}  
\end{align*}
is the $(a_n, k-s_na_n)$-torus link. 
Since the last isopoty happened in the complement of the braid axis of the last braid, we can apply the isopoty of Lemma~\ref{isopoty} in the backward direction to see that $K_1$ is the T-link $T((a_1, s_1a_1), \dots, (a_{n-1}, s_{n-1}a_{n-1}), (p, q))$.

Suppose that $q> a_{n-1}$. Then, by Lemma~\ref{braidindex}, $K_1$ has braid index equal to $q$.
So, $k-s_na_n = q$.
This implies that the sub-braid $$(\sigma_{a_n-1}\dots \sigma_{a_n-q+1})^{p-a_n}(\sigma_1\dots \sigma_{a_1-1})^{s_1a_1} \dots (\sigma_1\dots \sigma_{a_{n-1}-1})^{s_{n-1}a_{n-1}}$$ of $B_1$ is the trivial braid as $K_1$ is the torus link $T(a_n, q)$, but this is impossible as $a_{n-1}, \dots, a_1 > 1$. Thus, $q< a_{n-1}$. However, this contradicts the induction hypothesis.

Therefore, $K$ is never a torus link.
\end{proof}

\begin{lemma}\label{1}
Consider $n>1$. If $s_1>1$ or $a_2 \neq a_1 +1$, then the T-knot $$K = T((a_1, s_1a_1), (a_2, s_2a_2), \dots, (a_n, s_na_n), (p, bp+1))$$ is not a torus knot for $b>0$.
\end{lemma}

\begin{proof}Since $p$, $bp+1$ are coprime, $K$ is a knot.

Consider that $K $ is a torus knot. By \cite[Corollary 2.4]{Franks}, $K$ has braid index equal to $p$. By Lemma~\ref{Los}, the T-knot $$T((a_1, s_1a_1), (a_2, s_2a_2), \dots, (a_n, s_na_n), (p, 1))$$ is also a torus knot. By a Markov move, the last T-knot is equivalent to the T-knot $$T((a_1, s_1a_1), (a_2, s_2a_2), \dots, (a_n, s_na_n+1)).$$ By \cite[Corollary 2.4]{Franks}, this T-knot has braid index equal to $a_n$. Then, the T-knot $$T((a_1, s_1a_1), (a_2, s_2a_2), \dots, (a_{n-1}, s_{n-1}a_{n-1}), (a_n, 1)),$$ which is equivalent to the T-knot, by a Markov move, 
$$T((a_1, s_1a_1), (a_2, s_2a_2), \dots, (a_{n-1}, s_{n-1}a_{n-1}+1)),$$ is also a torus knot. Continuing with this argument, we obtain that the T-knot $$T((a_1, s_1a_1), (a_2, s_2a_2+1)),$$which is equivalent to the T-knot $T((a_1, s_1a_1), (s_2a_2+1, a_2))$ by Lemma~\ref{braidindex}, is also a torus knot, but this is a contradiction by \cite[Theorem 1.1]{LeeTorusknotsobtained} and \cite[Theorem 1.1]{Positively}.
\end{proof}

\begin{lemma}\label{2c}
Let $a_n < q < p$. Assume $n>1$ when $p$ and $q$ are coprime. 
If $p \neq bq +1$, or $s_1 > 1$, or $a_2 \neq a_1 +1$ for $b>0$,
then the T-link $$K = T((a_1, s_1a_1), \dots, (a_n, s_na_n), (p, q))$$ is not a torus link.
\end{lemma}

\begin{proof}
As $a_n < q < p$, the T-link $$K' = T((a_1, s_1a_1), \dots, (a_n, s_na_n), (q, p))$$ is equivalent to $K$ and has braid index equal to $q$ by Lemma~\ref{braidindex}.

Suppose that $K$ is a torus link. Then, $K$ is the $(q, k)$-torus link with $k>0$.
Let $r$ be a positive integer such that $0\leq p-rq< q$. The T-link $$K'_{-rq} = T((a_1, s_1a_1), \dots, (a_n, s_na_n), (q, p-rq))$$ is the torus link $T(q, k-rq)$ by Lemma~\ref{Los}. 

If $p-rq=0$, then $K'_{-rq}$ would be a splittable link. But, by Lemma 3.2 of \cite{HyperbolicTwistedTorusLinks}, torus links are irreducible. So, $K'_{-rq}$ can't be a torus link, which is a contradiction.

If $p-rq =1$, then $K'$ is a knot. However, by Lemma~\ref{1}, $K'$ can't be a torus knot. 

Thus, $p-rq>1$.
It is not possible that $a_n>p-rq$ by Lemma~\ref{q<a_n}. So, $a_n\leq p -rq$. 
If $a_n = p -rq$, then, by \cite[Lemma 2.7]{dePaivaPurcell:SatellitesLorenz}, 
$K'_{-rq}$ is equivalent to the T-link $$T((a_1, s_1a_1), \dots, (a_{n-1}, s_{n-1}a_{n-1}), (a_n, s_na_n + q)),$$which has braid index equal to $a_n$ by Lemma~\ref{braidindex}. Thus, $a_n = k-rq$ as $a_n \neq q$. This implies that $k = a_n+rq = p$. Hence, $a_1, \dots, a_n = 0$ as $K$ is the $(p, q)$-torus knot. This is a contradiction. 
If $a_n < p -rq$, then $K'_{-rq}$ has braid index equal to $p -rq$ by Lemma~\ref{braidindex}. So, $p -rq = q$ or $k-rq$. It is not possible that $p -rq = q$ since $0\leq p-rq< q$. If $p -rq = k-rq$, then $k = p$, which we have already ruled out.

Therefore, $K$ is not a torus link.
\end{proof}

\begin{proposition}\label{proposition1}
Let $p,q, a_1, \dots, a_n, s_1, \dots, s_n$ be positive integers such that $1<q<p$, $gcd(p, q)>1$, and $1<a_1<\dots<a_n<p$ with $a_i\neq q$. Then, the T-link $$K = T((a_1, s_1a_1), (a_2, s_2a_2), \dots, (a_n, s_na_n), (p, q))$$ is never a torus link.
\end{proposition}
\begin{proof}
If $q< a_n$, then $K$ is not a torus link by Lemma~\ref{q<a_n}. 

Consider now that $a_n < q $. It is not possible that $p = bq +1$ for $b>0$ otherwise $p$ and $q$ would be coprime.
Thus, from Lemma~\ref{2c}, $K$ is not a torus link.
\end{proof}

\begin{proposition}\label{proposition2}
Let $p,q, a_1, \dots, a_n, s_1, \dots, s_n$ be positive integers such that $1<q<p$, $gcd(p, q)=1$, and $1<a_1<\dots<a_n<p$ with $a_i\neq q$ and $n>1$.
If $q< a_n$, or $p \neq bq +1$, or $s_1 > 1$, or $a_2 \neq a_1 +1$ for $b>0$, then the T-knot $$T((a_1, s_1a_1), (a_2, s_2a_2), \dots, (a_n, s_na_n), (p, q))$$ is never a torus knot.
\end{proposition}
\begin{proof}
If $q< a_n$, then it follows from Lemma~\ref{q<a_n} that $K$ is not a torus knot. 
If $a_n < q $, under the hypotheses, $K$ is not a torus knot from Lemma~\ref{2c} either.
\end{proof}


\begin{named}{Theorem~\ref{maintheorem}}
Let $p,q, a_1, \dots, a_n, s_1, \dots, s_n$ be positive integers such that $1<q<p$ and $1<a_1<\dots<a_n<p$ with $a_i\neq q$.
\begin{itemize}
\item If $gcd(p, q)>1$, then the T-link $$T((a_1, s_1a_1), (a_2, s_2a_2), \dots, (a_n, s_na_n), (p, q))$$ is never a torus link;

\item Otherwise, if $q< a_n$, or $p \neq bq +1$, or $s_1 > 1$, or $a_2 \neq a_1 +1$ for $b>0$, then the T-knot $$T((a_1, s_1a_1), (a_2, s_2a_2), \dots, (a_n, s_na_n), (p, q))$$ is never a torus knot for $n>1$.
\end{itemize}
\end{named}
\begin{proof}
It follows from Propositions~\ref{proposition1} and ~\ref{proposition2}.
\end{proof}

\begin{corollary}\label{corollary}
Let $p,q, a_1, \dots, a_n, s_1, \dots, s_n$ be positive integers such that $1<q<p$, $gcd(p, q)=1$, and $1<a_1<\dots<a_n<p$ with $a_i\neq q$. If each $s_i$ is greater than one, then the T-knot $$K = T((a_1, s_1a_1), (a_2, s_2a_2), \dots, (a_n, s_na_n), (p, q))$$ is not a torus knot.
\end{corollary}

\begin{proof} 
If $n=1$, then the T-knot $K$ is not a torus knot by [\cite{LeeTorusknotsobtained}, Theorem 1.1]. Otherwise, the T-knot $K$ is not a torus knot by Theorem~\ref{maintheorem}.
\end{proof}



%

\bibliographystyle{amsplain}  

\bibliography{Torus_Lorenz_Links_obtained_by_full_twists}

\end{document}